\documentclass[12pt]{amsart}
\usepackage{amsmath}
\usepackage{amsthm}
\usepackage{amsfonts}
\usepackage{amssymb}
\newtheorem{Theorem}{Theorem}[section]
\newtheorem{Proposition}[Theorem]{Proposition}
\newtheorem{Lemma}[Theorem]{Lemma}

\theoremstyle{definition}

\theoremstyle{remark}

\numberwithin{equation}{section}

\newcommand{\R}{{\mathbb R}}
\newcommand{\C}{{\mathbb C}}

\newcommand{\tr}{{\textrm{\rm tr}\:}}

\renewcommand{\Im}{{\textrm{\rm Im}\:}}
\renewcommand{\Re}{{\textrm{\rm Re}\:}}

\begin{document}

\title[Reflectionless operators]{Reflectionless Dirac operators and matrix valued Krein functions}

\author{Christian Remling}

\address{Department of Mathematics\\
University of Oklahoma\\
Norman, OK 73019}
\email{christian.remling@ou.edu}
\urladdr{www.math.ou.edu/$\sim$cremling}

\date{December 1, 2024}



\begin{abstract}
I prove a sharp bound on reflectionless Dirac operators.
\end{abstract}
\maketitle
\section{Introduction}
This brief note is a spin-off of \cite{Remauto}. Its goal is to prove Theorem \ref{T1.2} below.
I originally tried to do this using the machinery of \cite{Remauto}, but I then realized
that the rather different methods from \cite{ClGes,GesSim} work much better for this. Incidentally, similar remarks apply to some of the results
of \cite{HMcBR,Remuniq}. So it seems to make sense to split this part off and present it separately here.

We consider Dirac equations
\begin{equation}
\label{Dirac}
Jy'(x) + W(x)y(x) = -zy(x) , \quad J = \begin{pmatrix} 0 & -1\\ 1&0 \end{pmatrix} ,
\end{equation}
and the associated operators $Ly=-Jy'-Wy$ on $L^2(\R;\C^2)$.
We assume that $W(x)\in\R^{2\times 2}$, $W(x)=W^t(x)$, $W\in L^1_{\textrm{loc}}(\R)$. Then $L$ is self-adjoint on its natural maximal domain
\[
D(L) = \{ y\in L^2(\R;\C^2): y \textrm{ absolutely continuous, } Jy'+Wy\in L^2 \} .
\]

The \textit{Titchmarsh-Weyl }$m$ \textit{functions }may be defined as
\begin{equation}
\label{defm}
m_{\pm}(z) = \pm y_{\pm}(0,z) ,
\end{equation}
and here $z\in\C^+=\{z\in\C : \Im z>0\}$ and $y_{\pm}(x,z)$ denotes the unique, up to a constant factor, solution $y$ of \eqref{Dirac} that
is square integrable on $\pm x>0$. On the right-hand side of \eqref{defm}, we also use the convenient convention of identifying
a vector $y=(y_1,y_2)^t\in\mathbb C^2$, $y\not= 0$, with the point $y_1/y_2\in\mathbb C_{\infty}$
on the Riemann sphere. So $m_{\pm}$ take values in $\C_{\infty}$, and in fact these functions are \textit{Herglotz functions, }that is,
they map the upper half plane $\mathbb C^+$ holomorphically back to itself.

Clearly, each of $m_{\pm}$ refers to one half line only. Of course, both of them combined contain all the information on the whole line problem,
so it must be possible to obtain a spectral representation of $L$ from $m_+$ and $m_-$, and usually one proceeds as follows:
combine $m_{\pm}$ into one matrix function
\begin{equation}
\label{defMmatrix}
M(z) = \frac{-1}{m_+(z)+m_-(z)} \begin{pmatrix} -2m_+(z)m_-(z) & m_+(z)-m_-(z) \\ m_+(z) - m_-(z) & 2 \end{pmatrix} .
\end{equation}
Then $M(z)$ is a \textit{matrix valued Herglotz function, }that is, $M(z)$ is holomorphic on $\C^+$ and we still have $\Im M(z)>0$ there,
where we now define $\Im M =(M-M^*)/(2i)$. Please see \cite{GesTse} for a comprehensive discussion of matrix valued Herglotz functions in general.

Our function has the additional properties $M=M^t$, so maps into what is often called the
\textit{Siegel upper half space, }and $\det M(z)=-1$ for all $z\in\C^+$.

The $M$ matrix provides a spectral representation of the Dirac operator $L$ in the sense that $L$ is unitarily equivalent to
multiplication by the variable in $L^2(\R, d\rho)$ on the natural domain of this operator, and here the (matrix valued) \textit{spectral measure }$\rho$
is the measure from the Herglotz representation of $M$:
\[
M(z) = A+Bz +\int_{-\infty}^{\infty} \left( \frac{1}{t-z}-\frac{t}{t^2+1}\right)\, d\rho(t)
\]

The function $M(z)$ or, equivalently, the pair of functions $m_{\pm}(z)$ does not determine $W(x)$ uniquely; such a one-to-one correspondence
can be obtained if $W$ is suitably normalized, which can be done in various ways. In this paper, I will work with the $\tr W=0$ normalization throughout.
We can then write
\begin{equation}
\label{1.7}
W(x) = \begin{pmatrix} p(x) & q(x) \\ q(x) & -p(x) \end{pmatrix} .
\end{equation}
These issues are discussed in more detail in the standard literature on the subject \cite{LevSar} and also in \cite[Section 2]{RemZ}, from a more
abstract point of view.

We say that $W$ or $L=L_W$ is \textit{reflectionless }on a Borel set $A\subseteq\R$ if
\begin{equation}
\label{refless}
m_+(t)= -\overline{m_-(t)}
\end{equation}
for (Lebesgue) almost every $t\in A$. Here, $m_{\pm}(t)\equiv\lim_{y\to 0+} m_{\pm}(t+iy)$; these limits exist at almost all $t\in\R$.
Reflectionless operators are important because they provide the basic building blocks for arbitrary operators with some absolutely continuous spectrum
\cite{RemAnn}, \cite[Ch.\ 7]{Rembook}.

We also define, for closed sets $E\subseteq\R$,
\begin{align*}
\mathcal R(E) & =\{ W: L_W\textrm{ is reflectionless on }E \} , \\
\mathcal R_0(E) & = \{ W\in\mathcal R(E) : \sigma(W)\subseteq E \} .
\end{align*}
We will focus on \textit{finite gap sets}
\begin{equation}
\label{defE}
E = \R \setminus \bigcup_{j=1}^n (a_j,b_j) ,
\end{equation}
with $a_1<b_1<a_2<\ldots < b_n$, though the arguments below can also handle more general situations.

As in the scalar case, we can take the (matrix valued) logarithm of a matrix valued Herglotz function to obtain a new Herglotz function.
We will review the details of the procedure in Section 2. The new function $\log M(z)$ has bounded imaginary part; in fact $0<\Im\log M(z)< \pi$
for a suitable choice of the logarithm,
and this implies that the representing measure of $\log M(z)$ is purely absolutely continuous. Its matrix valued density $\xi(t)\in\R^{2\times 2}$, $0\le\xi (t)\le 1$,
is called the \textit{Krein function }of $M(z)$. We have
\begin{equation}
\label{expM}
\log M(z) = A + \int_{-\infty}^{\infty} \left( \frac{1}{t-z}-\frac{t}{t^2+1}\right)\xi(t)\, dt ,
\end{equation}
with $A=\Re\log M(i)\in\R^{2\times 2}$, $A=A^t$.

The Krein function is a standard tool in the scalar case but it has not been used much for matrix valued Herglotz functions. It is easy to understand why this is so:
if, let's say, $\xi(t)=0$ or $\xi(t)=1$ in the scalar setting, then obviously $m(t)=|m(t)|e^{i\pi\xi(t)}$ is real. However,
if $\xi(t) = P$, a projection, in the matrix valued case, then we cannot automatically conclude that $M(t)$ is real even though $\xi(t)$ still has eigenvalues
$0$ and $1$. More precisely, $M(t)$ will be real only if $\Re\log M(t)$ commutes with $P$. So the converse of Proposition \ref{P1.1}(c) below fails badly.

These issues might deserve further investigation in a general framework. I will not try to do this here.
For my current purposes, the following straightforward properties of $\xi$ will be sufficient.
\begin{Proposition}
\label{P1.1}
(a) For any $W$, we have $\tr \xi(t) =1$, $t\in\R$.\\
(b) \cite{ClGes} $W\in\mathcal R(E)$ if and only if $\xi(t)=1/2$ on $t\in E$.\\
(c) For $t\notin\sigma(L)$, the Krein function is a projection:
\[
\xi(t) = P_{\alpha}=\begin{pmatrix} \cos^2\alpha & \sin\alpha\cos\alpha \\ \sin\alpha\cos\alpha & \sin^2\alpha \end{pmatrix}
\]
for some $\alpha=\alpha(t)$.
\end{Proposition}
When combined with the methods of \cite{ClGes,GesSim}, this will imply the following result. Recall here that if $W\in\mathcal R(E)$, $\tr W=0$,
then $W(x)$ is real analytic \cite[Theorem 4.1]{RemZ}, so it makes sense to evaluate this (matrix) function pointwise.
\begin{Theorem}
\label{T1.2}
If $W\in\mathcal R(E)$, $\tr W(x)=0$, with $E$ as in \eqref{defE}, then
\begin{equation}
\label{1.6}
\|W(x)\| \le \frac{1}{2} \sum_{j=1}^n (b_j-a_j) .
\end{equation}
Moreover, equality at a single $x_0\in\R$ implies that $W\in\mathcal R_0(E)$, and for any fixed $x=x_0\in\R$, there are such $W\in\mathcal R_0(E)$
for which \eqref{1.6} holds with equality.
\end{Theorem}
Here, $\| W(x)\|=\sqrt{ p^2(x)+q^2(x)}$ denotes the operator norm of $W(x)$.

If $n=1$, so $E=\R\setminus (a,b)$, then $\|W(x)\|=(b-a)/2$ for all $W\in\mathcal R_0(E)$ and $x\in\R$, and each $W\in\mathcal R_0(E)$ is constant.
This slightly strengthened version of Theorem \ref{T1.2} was obtained in \cite{RemZ}, by different methods. The present proof is simpler.

However, if $n>1$, a given $W\in\mathcal R_0(E)$ need not realize the bound \eqref{1.6} at any $x\in\R$
because orbits under the shift map $W(x)\mapsto W(x+a)$ need not be dense in $\mathcal R_0(E)$,
and (of course) the map $W\mapsto \|W(0)\|$ is no longer constant on $\mathcal R_0(E)$ when $n>1$.
\section{Matrix valued logarithms and Krein functions}
This section presents a quick review of material that can also be found in other sources such as \cite{GesTse} in one form or another,
with a view towards our needs here.

For a complex number
$w\in\Omega\equiv\C\setminus \{-iy: y\ge 0\}$, we define $\log w$ as the holomorphic function on this domain with $e^{\log w}=w$, $\log 1=0$.
So in particular $0<\Im\log w<\pi$ for $w\in\C^+$.

Having fixed this branch of the logarithm function, we then have available a well defined matrix $\log A$ for any $A\in\C^{2\times 2}$ with $\sigma(A)\subseteq\Omega$.
It satisfies $(\log A)v=(\log\lambda)v$ if $Av=\lambda v$. This property determines $\log A$ if $A$ is diagonalizable and could serve as the definition
of $\log A$ in this case. The general case can be
handled by approximation or a similar procedure, using the Jordan normal form. We have $e^{\log A}=A$, and here we define the matrix exponential
as usual by its power series.

In particular, since $\Im M(z)>0$, so $\sigma(M(z))\subseteq\C^+$, we may use this matrix logarithm for $A=M(z)$. For such matrices $A$,
with spectrum in the upper half plane, we can also compute $\log A$ as
\begin{equation}
\label{2.1}
\log A = \int_0^{\infty} \left( \frac{t}{t^2+1}-(t+A)^{-1} \right) \, dt ,
\end{equation}
as proposed in \cite{GesTse}. This formula works because the integral evaluates $\log w$ correctly if we plug in a number $w=A\in\C^+$. Representation \eqref{2.1}
is useful here because it shows that $\log M(z)$ is holomorphic on $z\in\C^+$ and $\Im\log M(z)>0$ there. Moreover, there is a similar formula for $i\pi-\log A$,
which will show that $\Im\log M(z)<\pi$.

As anticipated, we now define the \textit{Krein function }$\xi(t)$ as
\[
\xi(t) = \frac{1}{\pi}\lim_{y\to 0+}\Im\log M(t+iy) .
\]
The limit will exist for almost all $t\in\R$. Of course, $M$ has the same property, and if $M(t)=\lim M(t+iy)$ does exist, then also $\xi(t)=(1/\pi)\, \Im\log M(t)$
(though we cannot use \eqref{2.1} to compute the logarithm if $M(t)$ has real spectrum). Recall here that $\det M=-1$, so we still have $\sigma(M(t))\subseteq\Omega$.

The above discussion shows that $0\le \xi(t)\le 1$. Moreover, $\xi^t=\xi$ because $M$ and thus also $\log M$ have this property.

We can deduce the additional properties of $\xi$ listed in Proposition \ref{P1.1} most conveniently
from the following elementary description of the matrix logarithm.
\begin{Lemma}
\label{L2.1}
Suppose that $\sigma(A)\subseteq\Omega$. Then $B=\log A$ is the unique matrix satisfying $e^B=A$, $\sigma(B)\subseteq \{ z: -\pi/2<\Im z<3\pi/2 \}$.
\end{Lemma}
\begin{proof}[Sketch of proof]
The above discussion has shown that $B=\log A$ has these properties. To prove that there is only one such $B$ for a given $A$, notice that
$e^C$ is diagonalizable if and only if $C$ is. This observation immediately gives us uniqueness of $B$ when $A$ is diagonalizable. It also implies that if $A$ is not diagonalizable,
then $B$ must be of the form $B=\lambda+N$, $N^2=0$. In that case, since $\lambda=\lambda I$ and $N$ commute, $e^B=e^{\lambda}(1+N)$, and again
$B$ is determined by $A$.
\end{proof}
\section{Proof of Proposition \ref{P1.1}}
Part (a) is immediate from the formula
\[
-1=\det M(t)=\det e^{\log M(t)} = e^{\tr\log M(t)} ,
\]
since $\Im\tr \log M(t)=\pi\,\tr\xi(t)$ and $0\le\tr\xi(t)\le 2$.

To prove part (b), recall that \eqref{refless} is equivalent to
\begin{equation}
\label{reflessM}
\Re M(t) = 0 ;
\end{equation}
compare \cite[Proposition 7.8]{Rembook}, \cite[Lemma 8.1]{Teschl}. We can of course restrict our attention to those $t\in E$
for which $M(t)$ exists. Clearly, if $\xi(t)=1/2$, then $M(t)=e^{\log M(t)}=e^{A(t)+i\pi/2}=ie^{A(t)}$ satisfies \eqref{reflessM}. Conversely,
if \eqref{reflessM} holds, then $M(t)=iB$ with $B>0$ (recall again that $\det M(t)=-1$), so $B=e^A$ for some self-adjoint matrix $A$, and thus $\log M=A+i\pi/2$
by Lemma \ref{L2.1}.

Similarly, in the situation of part (c), $\Im M(t)=0$, so $M(t)=-\lambda P +(1/\lambda )(1-P)$ for some $\lambda>0$ and some projection $P$. Lemma \ref{L2.1}
thus shows that $\log M(t)=(\log\lambda+i\pi)P-\log\lambda(1-P)$, and in particular $\xi=P$ is the projection onto the negative eigenspace of $M(t)$.
\section{Proof of Theorem \ref{T1.2}}
If $W\in\mathcal R(E)$, $\tr W=0$, then $W(x)$ is real analytic \cite[Theorem 4.1]{RemZ}.
Moreover, $m_{\pm}(z)$ and $M(z)$ are holomorphic at $z=\infty$ \cite[Lemma 1.2]{RemZ}, and then \cite[eqn.\ (5.6)]{RemZ} says that
$m_+(z)=i -(q(0)+ip(0))/z + O(1/z^2)$, and here $p,q$ are the entries of $W$, as in \eqref{1.7}.
While this is not explicitly done in \cite{RemZ}, of course the same treatment applies to $m_-$, and it shows that similarly $m_-(z)=i+(q(0)-ip(0))/z +O(1/z^2)$.
In terms of $M$, this means that
\[
M(z) = i \left( 1- \frac{1}{z}W(0) + O(1/z^2) \right) .
\]
So, since the factor $i$ commutes with everything and $\log (1+A)$ can be computed in terms of its power series for $\|A\|<1$, we have
\begin{equation}
\label{4.1}
\log M(z) = \frac{i\pi}{2} -\frac{1}{z}W(0)+O(1/z^2)  .
\end{equation}
On the other hand, we can also obtain an asymptotic formula from \eqref{expM}. Write $\xi=\xi-1/2+1/2$ and recall that $\xi=1/2$ is the Krein function
of $M(z)=i$ and $\xi=1/2$ on $E$ by Proposition \ref{P1.1}(b). Hence
\begin{align*}
\log M(z) & = \frac{i\pi}{2} + \sum_{j=1}^n \int_{a_j}^{b_j} \frac{\xi(t)-1/2}{t-z}\, dt \\
& = \frac{i\pi}{2} -\frac{1}{z} \sum_{j=1}^n \int_{a_j}^{b_j} (\xi(t)-1/2)\, dt +O(1/z^2) ;
\end{align*}
notice that while \eqref{expM} would normally deliver an extra constant matrix $B$ on the right-hand side, we immediately see from
the asymptotic expansions that this equals zero here. Comparison with \eqref{4.1} then shows that
\begin{equation}
\label{4.2}
W(0) = \sum_{j=1}^n \int_{a_j}^{b_j} (\xi(t)-1/2)\, dt .
\end{equation}
This only gives $W$ at $x=0$, but it actually suffices to discuss the claims of Theorem \ref{T1.2} for $x=x_0=0$ since
$\mathcal R(E)$ and $\mathcal R_0(E)$ are invariant under shifts $W(x)\mapsto W(x+a)$.

Since $\|P-1/2\|=1/2$ for any projection, the bound of Theorem \ref{T1.2} is immediate from \eqref{4.2} and Proposition \ref{P1.1}(c).

To prove the final claims of Theorem \ref{T1.2}, observe that if $Q$ is in the closed convex hull of the projections $\{ P\}$ but not a projection itself,
then $\|Q-1/2\|<1/2$ since this matrix is self-adjoint and has trace zero and its norm is $\le 1/2$. Hence \eqref{4.2} also shows that $\|W(0)\|<(1/2)\sum (b_j-a_j)$
unless $\xi(t)\equiv P$ on $t\notin E$. We now finish the proof by showing that if $\xi$ is of this form, so $\xi=1/2$ on $E$, $\xi=P$ on $E^c$, then
$M(z)$, defined via \eqref{expM} with $A=0$, is the $M$ matrix of a $W\in\mathcal R_0(E)$.

Observe first of all that then $M$ is holomorphic near $z=\infty$ and $M(\infty)=i$; compare also \eqref{4.4} below.
This implies that $m_{\pm}$ have the same properties, and then we can
conclude as in the proof of \cite[Theorem 3.2]{RemZ} that $M$ is the $M$ matrix of a Dirac operator $L=L_W$.
Recall in this context that $m_{\pm}(z)$ can be recovered from $M(z)$, for example as the eigenvectors of $MJ$, as follows:
\[
M(z)J \begin{pmatrix} \pm m_{\pm}(z) \\ 1 \end{pmatrix} = \mp \begin{pmatrix} \pm m_{\pm}(z) \\ 1 \end{pmatrix} .
\]
Also, we clearly have $W\in\mathcal R(E)$, by Proposition \ref{P1.1}(b).

To show that $W\in\mathcal R_0(E)$, we again compute
\begin{align}
\nonumber
\log M(z) & = \frac{i\pi}{2} + \sum_{j=1}^n \int_{a_j}^{b_j} \frac{\xi(t)-1/2}{t-z}\, dt \\
\label{4.4}
& = \frac{i\pi}{2} + \sum_{j=1}^n \log \frac{b_j-z}{a_j-z} \left( P-\frac{1}{2}\right) .
\end{align}
This matrix is normal, with eigenvalues
\begin{equation}
\label{4.3}
\lambda_{\pm}(z)=\frac{i\pi}{2} \pm \frac{1}{2}\sum_{j=1}^n \log \frac{b_j-z}{a_j-z} ,
\end{equation}
and thus $M(z)$ will also be normal, with eigenvalues $e^{\lambda_{\pm}(z)}$.
We also see from \eqref{4.4} that $\log M(z)$ and thus also $M(z)$ itself have holomorphic continuations through each gap $(a_j,b_j)$. For $z=t\in (a_j,b_j)$,
the corresponding logarithm from \eqref{4.3} has a negative argument, while all the other ones have positive arguments. Thus $\Im\lambda_+(t)=\pi$,
$\Im\lambda_-(t)=0$, and it follows that $\Im M(t)=0$. (The potential objection that was mentioned in the introduction does not apply here since $M(t)$
is normal, so $\Re\log M(t)$, $\Im\log M(t)$ do commute.) Hence $(a_j,b_j)\cap\sigma(L)=\emptyset$; in other words,
$\sigma(L)\subseteq E$, and thus $L\in\mathcal R_0(E)$, as claimed.


\begin{thebibliography}{100}
\bibitem{ClGes}S.\ Clark and F.\ Gesztesy, Weyl-Titchmarsh $M$ function asymptotics, local uniqueness results,
trace formulas, and Borg-type theorems for Dirac operators,
\textit{Trans.\ Amer.\ Math.\ Soc.\ }\textbf{354 }(2001), 3475--3534.
\bibitem{GesSim}F.\ Gesztesy and B.\ Simon, The xi function,
\textit{Acta Math.\ }\textbf{176 }(1996), 49--71.
\bibitem{GesTse}F.\ Gesztesy and E.\ Tsekanovskii, On matrix-valued Herglotz functions,
\textit{Math.\ Nachr.\ }\textbf{218 }(2000), 61--138.
\bibitem{HMcBR}I.\ Hur, M.\ McBride, and C.\ Remling, The Marchenko representation of reflectionless Jacobi and Schrödinger operators,
\textit{Trans.\ Amer.\ Math.\ Soc.\ }\textbf{368 }(2016), 1251--1270.
\bibitem{LevSar}B.M.\ Levitan and I.S.\ Sargsjan, Sturm-Liouville and Dirac operators,
Mathematics and its Applications 59, Springer Science + Business Media, Dordrecht, 1991.
\bibitem{RemAnn}C.\ Remling, The absolutely continuous spectrum of Jacobi matrices,
\textit{Annals of Math.\ }\textbf{174 }(2011), 125--171.
\bibitem{Remuniq}C.\ Remling, Uniqueness of reflectionless Jacobi matrices and the Denisov-Rakhmanov Theorem,
\textit{Proc.\ Amer.\ Math.\ Soc.\ }\textbf{139 }(2011), 2175--2182.
\bibitem{Rembook}C.\ Remling, Spectral theory of canonical systems, de~Gruyter Studies in Mathematics 70,
Berlin/Boston, 2018.
\bibitem{Remauto}C.\ Remling, Reflectionless operators and automorphic Herglotz functions, in preparation.
\bibitem{RemZ}C.\ Remling and J.\ Zeng, Reflectionless Dirac operators and canonical systems,
preprint, \texttt{https://arxiv.org/abs/2410.20218.}
\bibitem{Teschl}G.\ Teschl, Jacobi operators and completely integrable nonlinear lattices,
Mathematical Monographs and Surveys 72, American Mathematical Society,
Providence, 2000.
\end{thebibliography}
\end{document}